\documentclass[english,11pt]{amsart}
\usepackage{amsopn,amsthm,amsfonts,amsmath,amssymb,amscd}
\usepackage{babel}
\usepackage{amssymb,tikz}

\newtheorem{Theorem}{Theorem}[section]
\newtheorem{Lemma}[Theorem]{Lemma}

\newtheorem{Def}[Theorem]{Definition}

\newenvironment{Proof*}{{\it Proof.}}

\newcommand{\NN}{\mathbb{N}}

\newcommand{\ZZ}{\mathbb{Z}}
\newcommand{\AG}{\mathbb{AG}}

\newcommand{\II}{\mathcal{I}}
\newcommand{\LL}{\mathcal{L}}
\newcommand{\MM}{\mathcal{M}}

\newcommand{\DEF}[1]{\emph{#1}}

\newcommand{\MD}{{\rm dim_M}}

\newcommand{\diam}[1]{{\rm diam}(#1)}

\begin{document}

\title{The metric dimension of the annihilating-ideal graph of a finite commutative ring}
\author{David Dol\v zan}
\date{\today}

\address{D.~Dol\v zan:~Department of Mathematics, Faculty of Mathematics
and Physics, University of Ljubljana, Jadranska 21, SI-1000 Ljubljana, Slovenia; e-mail: 
david.dolzan@fmf.uni-lj.si}

  \thanks{The author acknowledges the financial support from the Slovenian Research Agency  (research core funding No. P1-0222)}

\bigskip

\begin{abstract} 
We determine the metric dimension of the annihilating-ideal graph of a local finite commutative principal ring and a finite commutative principal ring with two maximal ideals. We also find the bounds for  the metric dimension of the annihilating-ideal graph of an arbitrary finite commutative principal ring.
\end{abstract}

\maketitle 

\section{Introduction}

\bigskip

One of the most important and active areas in algebraic combinatorics is the study
of graphs associated with different algebraic structures. This field has attracted many researchers during the past 30 years. 

One of the most basic concepts in the study of rings is the notion of a zero-divisor. Thus, in 1988, Beck \cite{beck88} first introduced the concept of the zero-divisor graph of a commutative ring. In 1999, Anderson and Livingston \cite{andliv99} made a slightly different definition of the zero-divisor graph in order to be able to investigate the zero-divisor structure of commutative rings. 

In \cite{AndBad08}, Anderson and Badawi introduced a similar notion of a total graph of a 
commutative ring $R$ as the graph with all elements of $R$ as vertices, and for distinct
$x, y \in R$, the vertices $x$ and $y$ are adjacent if and only if $x + y$ is a zero-divisor in $R$. 

In \cite {behboodi1} and \cite {behboodi2}, the authors defined and investigated an ideal-theoretic version of the zero-divisor graph, called the annihilating-ideal graph of a commutative ring. If $R$ is a commutative ring, then its annihilating-ideal graph, $\AG(R)$, is the graph whose vertices are the nonzero ideals of $R$ which have nontrivial annihilators, and in which there is an edge between two distinct ideals $I$ and $J$ if and only if $IJ = 0$. The annihilating-ideal graph is a natural object of study, since many facts about zero-divisors are easily expressed in the language of ideals. The annihilating-ideal graph eliminates some of the redundancy of the standard zero-divisor graph, and encodes some the the ring theoretic information more efficiently. 
On the other hand, since the annihilator-ideal graph reflects only information about the ideals of a ring, it cannot make certain fine distinctions between nonisomorphic rings which have very similar zero-divisor behavior. For example, $\AG(\ZZ_{16})$ and $\AG(\ZZ_{81})$ are isomorphic as graphs (each is a path of length 3) despite the fact that $\ZZ_{16}$ and $\ZZ_{81}$ are not isomorphic as rings.  Much study has been devoted to the subject of annihilating-ideal graphs (see for example \cite{aalipour, alini, curtis, nikandish, shaveisi, tamizh}).

The sequence of edges $x_0 - x_1$, $ x_1 - x_2$, ..., $x_{k-1} - x_{k}$ in a graph is called \emph{a path of length $k$}. The \DEF{distance} between vertices $x$ and $y$ is the length of the shortest
path between them, denoted by $d(x,y)$. 
The \DEF{diameter} $\diam{\Gamma}$ of the graph $\Gamma$ is the longest distance between any two vertices of the graph.
For an ordered subset $W=\{ w_1,w_2,\ldots,w_k\}$ of the vertex set of graph $G$ and a vertex $v$ of $G$ the $k$-vector
$r(v|W)=(d(v,w_1),d(v,w_2),\ldots,d(v,w_k))$ is called the \DEF{representation} of $v$ with respect to $W$. A set $W$ is called a \DEF{resolving set} for $G$ if distinct vertices of $G$ have distinct representations with respect to $W$. A resolving set of minimal cardinality for $G$ is a \DEF{basis} of $G$ and the cardinality of the basis is called the \DEF{metric dimension} of $G$, denoted by $\MD(G)$ \cite{chartrand}. Motivated by the problem of uniquely determining the location of an intruder in a network, Slater \cite{slater} introduced the concept of a metric dimension. The metric dimension was then studied by Harary and Melter \cite{harary} and it has appeared in various applications of graph theory, for example pharmaceutical chemistry \cite{cameron, chartrand}, robot navigation \cite{khuller}, combinatorial optimization \cite{sebo} and sonar and coast guard long range navigation \cite{slater}.
It turns out that determining the metric dimension of a graph remains a NP-complete problem even in special cases like bounded-degree planar graphs \cite{diaz}, or split graphs, bipartite graphs and their complements, line graphs of bipartite graphs \cite{epstein}.

Recently, Faisal et al. \cite{faisal} studied the metric dimension of the commuting graph of a dihedral group. In the ring setting, the metric dimension of a zero-divisor graph of a commutative ring was studied in \cite{pirzada}, while the metric dimension of a total graph of a finite commutative ring was studied in \cite{dolzan1}.

In this paper, we study the metric dimension of the annihilating-ideal graph of a commutative ring. Our main result is to exactly determine the metric dimension of the annihilating-ideal graph of a local finite commutative principal ring and a finite commutative principal ring with two maximal ideals. We also find the bounds for  the metric dimension of the annihilating-ideal graph of an arbitrary finite commutative principal ring (see Theorem \ref{main}).

\bigskip


\bigskip
\section{The metric dimension of the annihilating-ideal graph}
\bigskip

In this section, we shall investigate the metric dimension of the annihilating-ideal graph of a finite ring. 
As is well known (see for example \cite[Theorem VI.2]{mcdonald}), every finite commutative ring can be written as a direct sum of local rings. Also,
(see for example \cite[Proposition IV.7]{mcdonald}) the Jacobson radical $J$ of a finite ring $R$ is nilpotent. Thus there exists $n \in \NN$ such that $J^n=0$ and $J^{n-1} \neq 0$. We shall call this $n$ \emph{the nilpotency order} of $J$.
Recall, that a ring $R$ is called a \emph{chain ring} if all its ideals form a chain under inclusion. A well known result (see for example \cite[Theorem XVII.5]{mcdonald}) states that every finite principal local commutative ring is a chain ring.

Next, we shall also need the following definition.

\bigskip

\begin{Def}
Let $v$ be a vertex of a graph $G$. Then the \emph{open neighbourhood} of $v$ is $N(v)=\{u \in V(G); \text{ there exists an edge } uv \text { in G}\}$ and the \emph{closed neighbourhood} of $v$ us $N[v]=N(v) \cup \{v\}$.  
\end{Def}

%
%
%

\bigskip

We now have to investigate the distances between the ideals in $\AG(R)$. We know that any two ideals in $\AG(R)$ are at distance at most 3 (see \cite[Theorem 2.1]{behboodi1}). The next lemma describes the distances between ideals in $\AG(R)$ in more of a detail. 

\bigskip

\begin{Lemma}\label{distance2}
Let $R$ be a finite ring and $R \simeq R_1 \times R_2 \times \ldots \times R_n$ where for each $i=1,2,\ldots,n$, the ring $R_i$ is a local ring with the Jacobson radical $J_i$ of nilpotency order $n_i$. Let $\pi_s$ denote the canonical projection of $R$ onto $R_s$.  Let $M, N$ be distinct ideals in $\AG(R)$. Then $d(M,N) = 3$  if and only if $R_s \in \{\pi_s(M), \pi_s(N)\}$ for all $s=1,2, \ldots, n$ and there exists $k \in \{1,2, \ldots, n\}$ such that $0 \notin \{\pi_k(M), \pi_k(N)\}$. 
\end{Lemma}
\begin{proof}
 Let $I_{s,k}$ be an ideal of $R$ such that $\pi_i(I_{s,k})=0$ for all $i \neq s$ and $\pi_s(I_{s,k})=J_s^k$.

If $0 \in \{\pi_k(M), \pi_k(N)\}$ for all $k  \in \{1,2, \ldots, n\}$, then $d(M, N)=1$. If there exists $s \in \{1,2,\ldots,n\}$ such that $R_s \notin \{\pi_s(M), \pi_s(N)\}$ then $M$ and $N$ are both connected to $I_{s,n_s-1}$, so $d(M,N) \leq 2$. Conversely, suppose that $R_s \in \{\pi_s(M), \pi_s(N)\}$ for all $s=1,2,\ldots,n$ and there exists $k \in \{1,2, \ldots, n\}$ such that $0 \notin \{\pi_k(M), \pi_k(N)\}$. The latter condition yields that $d(M,N) \neq 1$. If $d(M,N) = 2$ then there exists $T \in \AG(R)$ such that $T$ is connected to $M$ and $N$. However, this implies that $\pi_s(T)=0$ for all $s=1,2,\ldots,n$ which is a contradiction.
\end{proof}

\bigskip

Suppose $R$ is a a finite commutative principal ring with exactly $n$ maximal ideals.
In the next theorem, we exactly determine the metric dimension of $\AG(R)$ in case $n \leq 2$ and find the bounds for metric dimension of $\AG(R)$ in case $n \geq 3$.

\bigskip

\begin{Theorem}\label{main}
Let $R$ be a finite commutative principal ring. Then $R \simeq R_1 \times R_2 \times \ldots \times R_n$ where for each $i=1,2,\ldots,n$, the ring $R_i$ is a local chain ring with the Jacobson radical $J_i$ of nilpotency order $n_i$. Moreover, the following statements hold.
\begin{enumerate}
\item
If $n=1$ then $\MD(\AG(R)) =\lfloor{\frac{n_1-1}{2}}\rfloor$. 
\item
If $n =2$ then $\MD(\AG(R)) = n_1 + n_2 -2+\epsilon$, where $\epsilon=1$ if $R_1, R_2$ are both fields and $\epsilon=0$ otherwise.
\item
If $n \geq 3$ then $\left(\sum_{i=1}^n{n_i}\right)-n+\epsilon \leq \MD(\AG(R)) \leq \sum_{i=1}^n{n_i}$, where $\epsilon=0$ if none of $R_1,R_2,\ldots,R_n$ are fields and $\epsilon=\lceil\log_2(\beta)\rceil$, where $\beta$ denotes the number of rings $R_i$ that are fields.
\end{enumerate}
\end{Theorem}
\begin{proof}
By the comments at the beginning of this section, we know that $R \simeq R_1 \times R_2 \times \ldots \times R_n$ where for each $i=1,2,\ldots,n$, the ring $R_i$ is a local chain ring with the Jacobson radical $J_i$ of nilpotency order $n_i$.

Suppose first that $n=1$. Since $R$ is a local chain ring, the set of vertices in $\AG(R)$ is equal to $\{J, J^2, \ldots, J^{n_1-1} \}$, where $J=J_1$ denotes the Jacobson radical of $R$ of nilpotency order $n_1$. If $n_1 \leq 2$, then $\AG(R)$ is either an empty graph or a singleton, so we can assume that $n_1 \geq 3$. By Lemma \ref{distance2}, we have that $\diam{\AG(R)} \leq 2$. This implies that if $X$ is a resolving set for $\AG(R)$, no two vertices outside $X$ can have exactly the same neighbours in $X$. Now, for each $k \in \{1,2,\ldots,n_1-2\}$, observe that $N[J^{k+1}] = N[J^{k}] \cup \{ J^{n_1-k-1}\}$. 
This implies that any vertex $J^k$ in the resolving set differentiates between two sets of vertices: the vertices $\{J, \ldots, J^{n_1-k-1}\}$ that are not connected to $J^k$ and the vertices $\{J^{n_1-k}, \ldots, J^{n_1-1}\}$ that are connected to $J^k$. The resolving set has to contain all vertices such that any two vertices are differentiated for by at least one of the vertices in the resolving set. So, if the resolving set contains $s$ vertices, then the remaining $n_1-1-s$ vertices have to be differentiated into singletons by putting $s$ differentiating markers between them. But to separate  $n_1-1-s$ vertices into singletons one needs to put down at least $n_1-2-s$ differentiaing markers. Therefore,
$s \geq n_1-2-s$ and thus $s \geq \frac{n_1-2}{2}$. Since $s$ is an integer, $s \geq \lceil{\frac{n_1-2}{2}}\rceil=\lfloor{\frac{n_1-1}{2}}\rfloor$. We have thus proved that each resolving set contains at least $\lfloor\frac{n_1-1}{2}\rfloor$ elements, so $\MD(\AG(R)) \geq \lfloor{\frac{n_1-1}{2}}\rfloor$.
Conversely, we prove that $X=\{J, J^2, \ldots, J^{\lfloor{\frac{n_1-1}{2}}\rfloor}\}$ is a resolving set for a local chain ring $R$ with the Jacobson radical $J$ of nilpotency order $n_1$. If $n_1$ is even, then  $X=\{J, J^2, \ldots, J^{\frac{n_1}{2}-1}\}$ and each $J^{\frac{n_1}{2}-1+k}$ is connected exactly to the first $k-1$ elements from $X$. On the other hand, if $n_1$ is odd, then  $X=\{J, J^2, \ldots, J^{\frac{n_1-1}{2}}\}$ and each $J^{\frac{n_1-1}{2}+k}$ is connected exactly to the first $k$ elements from $X$. Thus, $\MD(\AG(R)) \leq \lfloor{\frac{n_1-1}{2}}\rfloor$.

Suppose now that $n \geq 2$. 
Let $\pi_s$ denote the canonical projection of $R$ onto $R_s$.  Choose $s \in \{1,2,\ldots,n\}$ and $k \in \{0,1,\ldots,n_s-1\}$ and let $I_{s,k}$ be an ideal of $R$ such that $\pi_i(I_{s,k})=0$ for all $i \neq s$ and $\pi_s(I_{s,k})=J_s^k$.

Consider first the case $n=2$. Now, we have three possibilites.

If $R_1$ and $R_2$ are both fields, then define $X=\{I_{1,0}\}$ and observe that $\vert X \vert = 1 = n_1+n_2-2 + 1$. Since $\AG(R)$ has only two vertices, $X$ is obviously a resolving set.

If exactly one of the rings $R_1$ and $R_2$ is a field (say, $R_2$) then define $X=\{I_{1,k}; k \in \{0,1,\ldots,n_1-2\} \}$ and observe that $\vert X \vert = n_1-1 = n_1+n_2-2$. Choose $M \neq L \in \AG(R)$. Then $M=J_1^{m_1} \times J_2^{m_2}$ and $L=J_1^{l_1} \times J_2^{l_2}$
for some integers $0 \leq m_1, m_2 \leq n_1$ and $0 \leq l_1, l_2 \leq 1$. Since $M \neq L$, we have $m_1 \neq l_1$ or $m_2 \neq l_2$. If $m_1=0$ then $M=R_1 \times 0$, so $M \in X$ and we can reason similarly for $l_1$. So, we can assume that $m_1, l_1 \geq 1$. Suppose first that $m_1 \neq l_1$. We can also suppose without loss of generality that $m_1 < l_1$. Since $m_1 \geq 1$, we have $d(M,I_{1,n_1-m_1-1}) > 1$, while $d(L,I_{1,n_1-m_1-1}) = 1$.  So, suppose now that $m_1=l_1$. Then we can assume without loss of generality that $M=J_1^{m_1} \times 0$ and $L=J_1^{m_1} \times R_2$. But then by Lemma \ref{distance2}, we get $d(M,I_{1,0}) = 2$, while $d(L,I_{1,0}) = 3$.
This proves that $X$ is a resolving set.

Finally, suppose neither of $R_1$ and $R_2$ are fields.
Define $X=\{I_{s,k}; s \in \{1,2\}$ and $k \in \{0,1,\ldots,n_s-2\} \}$ and observe that $\vert X \vert = n_1-1+n_2-1 = n_1+n_2-2$.
Again, choose $M \neq L \in \AG(R)$. Then $M=J_1^{m_1} \times J_2^{m_2}$ and $L=J_1^{l_1} \times J_2^{l_2}$
for some integers $0 \leq m_1, m_2 \leq n_1$ and $0 \leq l_1, l_2 \leq n_2$. Since $M \neq L$, we have $m_1 \neq l_1$ or $m_2 \neq l_2$.
We can assume without loss of generality that $m_1 < l_1$. If $m_1=0$ then $M=R_1 \times J_2^{m_2}$ and this implies by Lemma \ref{distance2} that $d(M,I_{2,0}) = 3$, while $d(L,I_{2,0}) \leq 2$. On the other hand, if $m_1 \geq 1$, then $d(M,I_{1,n_1-m_1-1}) > 1$, while $d(L,I_{1,n_1-m_1-1}) = 1$.
This proves again that $X$ is a resolving set.

Since in all three cases, we have $\vert X \vert = n_1+n_2-2 + \epsilon$, where $\epsilon=1$ if $R_1, R_2$ are both fields and $\epsilon=0$ otherwise. Therefore, we have proved that in the case $n=2$, we have $\MD(\AG(R)) \leq n_1+n_2-2+\epsilon$.

Suppose now that $n \geq 3$. Define $X=\{I_{s,k}; s \in \{1,2,\ldots,n\}$ and $k \in \{0,1,\ldots,n_s-1\} \}$ and observe that $\vert X \vert = \sum_{i=1}^n{n_i}$. 
Choose $M \neq L \in \AG(R)$. Then $M=J_1^{m_1} \times J_2^{m_2} \times \ldots \times J_n^{m_n}$ and $L=J_1^{l_1} \times J_2^{l_2} \times \ldots \times J_n^{l_n}$
for some integers $m_1, m_2, \ldots, m_n, l_1, l_2, \ldots, l_n$ with $0 \leq m_i,l_i \leq n_i$ for all $i \in \{1,2,\ldots,n\}$. Since $M \neq L$, we have $i$ such that $m_i \neq l_i$. We can suppose without any loss of generality that $m_i < l_i$. Therefore $m_i \leq n_i-1$ and $d(M,I_{i,n_i-m_i-1}) > 1$, while $d(L,I_{i,n_i-m_i-1}) = 1$.  
This proves that $X$ is a resolving set and thus $\MD(\AG(R)) \leq \sum_{i=1}^n{n_i}$.
We have thus established all the upper bounds stated in the theorem in the case $n \geq 2$.

We now have to find the lower bounds.
Therefore, let $n \geq 2$ and let $X$ be the resolving set for $\AG(R)$.  For $s \in \{1,2,\ldots,n\}$ and $k \in \{0,1,\ldots,n_s-1\}$ , let $L_{s,k}$ be an ideal of $R$ such that $\pi_i(L_{s,k})=R_i$ for all $i \neq s$ and $\pi_s(L_{s,k})=J_s^{n_s-k}$. 
If $I_{s,k}, \ldots, I_{s,k+t} \notin X$ for some $0 \leq t \leq n_s-2-k$, then by Lemma \ref{distance2}, we have  $d(L_{s,k}, I)=\ldots=d(L_{s,k+t+1},I)$ for all ideals $I$ in $X$. Since again by Lemma \ref{distance2}, we have
$d(L_{s,k},L_{s,k+r})=2$ for all $1 \leq r \leq n_s-1-k$, then we can conclude that at least $t$ of the ideals $L_{s,k}, \ldots, L_{s,k+t+1}$ are elements of $X$ .
Consequently, for every $s \in \{1,2,\ldots,n\}$ we have 
\begin{equation}\tag{a}\label{ILeq}
\vert \{I_{s,k}; I_{s,k} \notin X, k=0,1,\ldots, n_s-2\} \vert \leq \vert \{L_{s,k}; L_{s,k} \in X, k=0,1,\ldots, n_s-1\} \vert.
\end{equation}

If $n=2$, define $\II = \emptyset$ and $\LL=\emptyset$ and if $n \geq 3$, define the sets $\II = \bigcup_{s=1,k=0}^{n,n_s-2}{I_{s,k}}$, $\LL =  \bigcup_{s=1,k=0}^{n,n_s-1}{L_{s,k}}$.
Let $N_F = \{s  \in \{1,2,\ldots, n\}; R_s $ is a field$\}$, so $\beta = \vert N_F \vert$. If $\beta \leq 1$ then define $\MM = \emptyset$. Assume now that $\beta \geq 2$. For each subset $N_F' \subseteq N_F$ of cardinality $1 \leq t \leq \beta$, we shall prove that there exists an integer $k \geq \lceil\log_2(t)\rceil$ and distinct ideals $M_1, M_2, \ldots, M_{k} \in X$ such that $\{M_1, M_2, \ldots, M_{k}\} \cap (\II \cup \LL) = \emptyset$. We shall prove this by induction on $t$.  If $t=1$, there is obviously nothing to prove.
Now, assume that the statement holds for $t$ and let us prove it holds also for $t+1$. Choose $s_1, s_2 \in N_F'$. The fact that $X$ is a resolving set gives us the following two possibilities: either at least one of the ideals $I_{s_1,0}$ and $I_{s_2,0}$ is in $X$, or there exists an ideal $M \in X$ such that $d(I_{s_1,0},M) \neq d(I_{s_2,0},M)$. Obviously, $I_{s_1,0}, I_{s_2,0}, M \notin \II \cup \LL$, so in each case we have found an ideal in $X$ that is not in $\II \cup \LL$. Denote this ideal by $N$. Now, $\pi_s(N) \in \{0,R_s\}$ for all $s \in N_F'$. Let $N_F'(1) = \{s; \pi_s(N)=0\}$ and $N_F'(2)=N_F' \setminus N_F'(1)$. Denote by $f_1$ and $f_2$ the cardinalities of $N_F'(1)$ and $N_F'(2)$, respectively. Suppose that $f_1 \geq f_2$ (in the other case the proof is symmetrical). Thus, $f_1 \geq \frac{t}{2}$ and by the induction hypothesis there exists an integer $k \geq \lceil\log_2(\frac{t}{2})\rceil$ and distinct ideals $M_1, M_2, \ldots, M_{k} \in X$ such that $\{M_1, M_2, \ldots, M_{k}\} \cap (\II \cup \LL) = \emptyset$. We define $M_{k+1} = N$.  Since $d(I_{s_1',0},N) = d(I_{s_2',0},N)$ for all $s_1',s_2' \in N_F'(1)$, the ideal $N$ is distinct from all the ideals $M_1, M_2, \ldots, M_{k}$.  Thus, we have found $1+k \geq 1 + \lceil\log_2(\frac{t}{2})\rceil \geq \lceil\log_2(t)\rceil$ ideals in $X \setminus \II \cup \LL$ and the statement follows.
We now use this statement for the set $N_F$ of cardinality $\beta$ and let $\MM$ denote the set of ideals obtained in this way, where $\vert \MM \vert \geq \lceil\log_2(\beta)\rceil$.  

Now define $X' = (X \cup \II \cup \MM) \setminus \LL = (X \cup \II) \setminus \LL$. 
Again, suppose first that $n=2$. If $R_1$ is not a field, then $X$ contains by the inequality (\ref{ILeq}) at least $n_1-1$ of the ideals from the set $\{I_{1,0}, I_{1,1}, \ldots, I_{1,n_1-2}, L_{1,0}, L_{1,1},\ldots,L_{1,n_1-1}\}$ (obviously this also holds in case $R_1$ is a field, since then $n_1-1=0$).
Similarly, we find $n_2-1$ ideals from $X$ from the set $\{I_{2,0}, I_{2,1}, \ldots, I_{2,n_1-2}, L_{2,0}, L_{2,1},\ldots,L_{2,n_1-1}\}$. However, if $I_{1,0} = L_{2,0} \in X$ then again by the inequality (\ref{ILeq}) we know that there are $n_1-2$ ideals in $X$ from the set $\{I_{1,1}, I_{1,2}, \ldots, I_{1,n_1-2}, L_{1,1}, L_{1,2},\ldots,L_{1,n_1-1}\}$. Similarly, if $I_{2,0} = L_{1,0} \in X$ then we find $n_2-2$ ideals in $X$ from the set  $\{I_{2,1}, I_{2,2}, \ldots, I_{2,n_1-2}, L_{2,1}, L_{2,2},\ldots,L_{2,n_1-1}\}$. This always yields at least  $n_1+n_2-2$ ideals in $X$, so $\vert X \vert \geq n_1+n_2-2$ if at least one of the rings $R_1$ and $R_2$ is not a field, and $\vert X \vert \geq 1 = n_1 + n_2 - 2 + 1$ otherwise.  Therefore,
$\MD(\AG(R)) \geq n_1+n_2-2+\epsilon$, where $\epsilon = 1$ if $R_1, R_2$ are both fields and $\epsilon=0$ otherwise.

If $n \geq 3$, we have $\vert X' \vert \leq \vert X \vert$ by the inequality (\ref{ILeq}). The fact that $n \geq 3$ implies that $\II \cap \LL = \emptyset$, so by the construction of $X'$, we have $I_{s,k} \in X'$ for all  $s \in \{1,2,\ldots,n\}$ and $k \in \{0,1,\ldots,n_s-2\}$. Thus, since $\II \cap \MM = \emptyset$, we arrive at $\vert X' \vert \geq \left(\sum_{i=1}^n{n_i}\right)-n+\epsilon$, where $\epsilon=\lceil\log_2(\beta)\rceil$. 
This yields $\vert X \vert \geq \vert X' \vert \geq \left(\sum_{i=1}^n{n_i}\right)-n+\epsilon$, therefore
$\MD(\AG(R)) \geq \left(\sum_{i=1}^n{n_i}\right)-n+\epsilon$ and thus the theorem holds.
\end{proof}

\bigskip

%
%



\end{document}